\documentclass[11pt]{amsart}
\usepackage{amsfonts, amssymb, amsmath, amsthm, hyperref, color, float,enumerate}
\usepackage[latin1]{inputenc}

\theoremstyle{plain}
\newtheorem{teo}{Theorem}

\newtheorem{lema}[teo]{Lemma}

\theoremstyle{definition}

\theoremstyle{remark}
\newtheorem{obs}{Remark}

\hypersetup{
	colorlinks = true,
	linkcolor = cyan,
	anchorcolor = blue,
	citecolor = magenta,
	filecolor = blue,
	urlcolor = lila
}

\numberwithin{equation}{section}
\allowdisplaybreaks
\newcommand{\R}{\mathbb{R}}

\newcommand{\supp}{\mathrm{supp}}

\begin{document}

	\title[Fefferman-Stein type Mixed inequalities]{Mixed inequalities of Fefferman-Stein type for singular integral operators}

	\author[F. Berra]{Fabio Berra}
	\address{CONICET and Departamento de Matem\'{a}tica (FIQ-UNL),  Santa Fe, Argentina.}
	\email{fberra@santafe-conicet.gov.ar}
	
	\author[M. Carena]{Marilina Carena}
	\address{CONICET and Departamento de Matem\'{a}tica (FIQ-UNL),  Santa Fe, Argentina.}
	\email{marilcarena@gmail.com}
	
	\author[G. Pradolini]{Gladis Pradolini}
	\address{CONICET and Departamento de Matem\'{a}tica (FIQ-UNL),  Santa Fe, Argentina.}
	\email{gladis.pradolini@gmail.com}
	
	\thanks{The authors were supported by CONICET, UNL and ANPCyT}
	
	
	
	\begin{abstract}
		We give Feffermain-Stein type inequalities related to mixed estimates for Calderón-Zygmund operators. More precisely, given $\delta>0$, $q>1$, $\varphi(z)=z(1+\log^+z)^\delta$, a nonnegative and locally integrable function $u$ and $v\in \mathrm{RH}_\infty\cap A_q$, we prove that the inequality
		\[uv\left(\left\{x\in \mathbb{R}^n: \frac{|T(fv)(x)|}{v(x)}>t\right\}\right)\leq \frac{C}{t}\int_{\mathbb{R}^n}|f|\left(M_{\varphi, v^{1-q'}}u\right)M(\Psi(v))\]
		holds with $\Psi(z)=z^{p'+1-q'}\mathcal{X}_{[0,1]}(z)+z^{p'}\mathcal{X}_{[1,\infty)}(z)$, for every $t>0$ and every $p>\max\{q,1+1/\delta\}$. This inequality provides a more general version of mixed estimates for Calderón-Zygmund operators proved in \cite{CruzUribe-Martell-Perez}. It also generalizes the Fefferman-Stein estimates given in \cite{P94} for the same operators.
		
		We further get similar estimates for operators of convolution type with kernels satisfying an $L^\Phi-$Hörmander condition, generalizing some previously known results which involve mixed estimates and Fefferman-Stein inequalities for these operators.
	\end{abstract}
	
	\maketitle

	\section{Introduction and main results}
	
	In 1985, E. Sawyer proved an endpoint estimate on the real line for the Hardy-Littlewood maximal operator $M$ which involved two different weights (see \cite{Sawyer}). More precisely, if $u,v\in A_1$ then the inequality
	\begin{equation}\label{eq: intro - mixta de Sawyer}
	uv\left(\left\{x\in\mathbb{R}: \frac{M(fv)(x)}{v(x)}>t\right\}\right)\leq \frac{C}{t}\int_{\mathbb{R}}|f(x)|u(x)v(x)\,dx
	\end{equation}
	holds for every positive $t$. This estimate, which can be seen as the weak $(1,1)$ type inequality of $Sf=M(fv)/v$ with respect to the measure $d\mu(x)=u(x)v(x)\,dx$, allowed to give an alternative proof of the boundedness of $M$ in $L^p(w)$ when $w\in A_p$, a result due to Muckenhoupt in \cite{Muck72}. Different extensions of \eqref{eq: intro - mixta de Sawyer} were obtained, see for example \cite{CruzUribe-Martell-Perez}, \cite{O-P} and \cite{L-O-P} for $M$ and Calderón-Zygmund operators (CZO), \cite{Berra-Carena-Pradolini(M)} for commutators of CZO, \cite{Berra-Carena-Pradolini(J)} for fractional operators, \cite{Berra} and \cite{Berra-Carena-Pradolini(MN)} for generalized maximal operators associated to Young functions.
	
	On the other hand, in \cite{FS71} it was shown that if $w$ is a nonnegative and locally integrable function and $1<p<\infty$ then
	\[\int_{\mathbb{R}^n} (Mf(x))^pw(x)\,dx\leq C\int_{\mathbb{R}^n}|f(x)|^pMw(x)\,dx,\]
	where $C$ depends only on $p$. We shall refer to this type of estimate as Fefferman-Stein inequalities. Regarding CZO, a first result due to Córdoba and Fefferman \cite{CF76} established that if $w$ is a nonnegative and locally integrable function then
	\[\int_{\mathbb{R}^n} |Tf(x)|^pw(x)\,dx\leq C_{p,r}\int_{\mathbb{R}^n}|f(x)|^pM(M_r w)(x)\,dx,\]
	for $1<p,r<\infty$. Later on, Wilson  improved the estimate above in \cite{W89} for rough singular integrals, obtaining the operator $M^2$ on the right-hand side, which is pointwise lesser than $M(M_r)$. Another estimates for CZO were proved by Pérez in \cite{P94}, where the maximal operators involved are related to Young functions satisfying certain properties  (see Theorem~\ref{teo: tipo fuerte de T con peso arbitrario}).
	
	Concerning Fefferman-Stein estimates for mixed inequalities, in \cite{Berra-Carena-Pradolini(M)} we prove a result involving a radial power function $v$ that fails to be locally integrable in $\mathbb{R}^n$ and a nonnegative function $u$ given by 
	\[uw\left(\left\{x\in\mathbb{R}^n: \frac{M_{\Phi}(fv)(x)}{v(x)}>t\right\}\right)\leq C\int_{\mathbb{R}^n}\Phi\left(\frac{|f|v}{t}\right)Mu,\]
	where $\Phi$ is a Young function of $L\log L$ type and $w$ depends on $v$ and $\Phi$.
	 This result generalizes a previous estimate proved in \cite{O-P}, where  the authors exhibit a counterexample for the Hardy-Littlewood maximal operator $M$, showing that the estimate above fails to be true for pairs $(u,M^2u)$ and $v$ in $\mathrm{RH}_\infty$.
	
	In this paper we study Fefferman-Stein inequalities for mixed estimates involving CZO. We shall be dealing with a linear operator $T$, bounded on $L^2=L^2(\R^n)$ and such that for $f\in L^2$ with compact support we have the representation
	\begin{equation}\label{eq: representacion integral de T}Tf(x)=\int_{\R^n}K(x-y)f(y)\,dy ,\quad\quad x\notin \supp f,\end{equation}
	where $K\colon\mathbb{R}^n\backslash\{0\}\to\mathbb{C}$ is a measurable function defined away from the origin. We say that $T$ is a CZO if $K$ is a standard kernel, which means that it satisfies a size condition given by 
	\[|K(x)|\lesssim \frac{1}{|x|^n},\]
	and the following smoothness condition also holds 
	\begin{equation}\label{eq:prop del nucleo}
	|K(x-y)-K(x-z)|\lesssim \frac{|x-z|}{|x-y|^{n+1}},\quad \textrm{ if } |x-y|>2|y-z|.
	\end{equation}
	The notation $A\lesssim B$ means, as usual, that there exists a positive constant $c$ such that $A\leq cB$. When $A\lesssim B$ and $B\lesssim A$ we shall write $A\approx B$.
	
	We are now in a position to state our main results.
	\begin{teo}\label{teo: acotacion mixta general para T}
		Let $0\leq u\in L^1_{\rm loc}$, $q>1$ and $v\in \mathrm{RH}_\infty\cap A_q$. Let $T$ be a CZO, $\delta>0$ and $\varphi(z)=z(1+\log^+z)^\delta$. Then for every $p>\max\{q,1+1/\delta\}$ the inequality
		\[uv\left(\left\{x\in \mathbb{R}^n: \frac{|T(fv)(x)|}{v(x)}>t\right\}\right)\leq \frac{C}{t}\int_{\mathbb{R}^n}|f(x)|M_{\varphi, v^{1-q'}}u(x)M(\Psi(v))(x)\,dx\]
		holds for every positive $t$ and every bounded function $f$ with compact support, where $\Psi(z)=z^{p'+1-q'}\mathcal{X}_{[0,1]}(z)+z^{p'}\mathcal{X}_{[1,\infty)}(z)$.   
	\end{teo}

When $v=1$ the theorem above gives the result proved in \cite{P94} for CZO. It also corresponds to the case $m=0$ of the commutator operator given in \cite{PP01}. This type of estimate is not only an extension of the well-known weak endpoint inequality for the operator $T$ but also provides an estimate of the type $(u,\tilde{M}u)$ for mixed inequalities, where $\tilde{M}$ is an adequate maximal function.

We shall also consider operators as in \eqref{eq: representacion integral de T} associated to kernels with less regularity properties,  which appeared in the study of Coifman type estimates for these operators. It was proved in \cite{MPT05} that the classical Hörmander condition  on the kernel fails to achieve the desired estimate (see also \cite{Lorente-Riveros-delaTorre05}). We now introduce the notation related to this topic. Given a Young function $\varphi$, we denote
\[\|f\|_{\varphi,|x|\sim s}=\left\|f\mathcal{X}_{|f|\sim s}\right\|_{\varphi, B(0,2s)}\]
where $|x|\sim s$ means that $s<|x|\leq 2s$ and $\|\cdot\|_{\varphi,B(0,2s)}$ denotes the Luxemburg average over the ball $B(0,2s)$ (see Section~\ref{seccion: preliminares} for further details). 

We say that $K$ satisfies the $L^{\varphi}-$Hörmander condition, and we denote it by $K\in H_\varphi$, if there exist constants $c\geq 1$ and $C_\varphi>0$ such that the inequality
\begin{equation}\label{eq: condicion Hormander}
\sum_{k=1}^\infty (2^kR)^n\|K(\cdot-y)-K(\cdot)\|_{\varphi,|x|\sim 2^kR}\leq C_\varphi
\end{equation}
holds for every $y\in\mathbb{R}^n$ and $R>c|y|$. When $\varphi(t)=t^r$, $r\geq 1$, we write $H_\varphi=H_r$. 

In \cite{Lorente-Martell-Perez-Riveros} the authors prove certain Fefferman-Stein inequalities for these type of operators. Concretely, if $\Phi$ is a Young function and there exists $1<p<\infty$ and Young functions   $\eta,\varphi$ such that $\eta\in B_{p'}$ and $\eta^{-1}(z)\varphi^{-1}(z)\lesssim \tilde\Phi^{-1}(z)$ for $z\geq z_0\geq 0$, then the inequality
\begin{equation}\label{eq: F-S para T Hormander}
w\left(\left\{x\in \mathbb{R}^n: |T(fv)(x)|>t\right\}\right)\leq \frac{C}{t}\int_{\mathbb{R}^n}|f(x)|M_{\varphi_p}w(x)\,dx
\end{equation}
holds with $\varphi_p(z)=\varphi(z^{1/p})$, and where $\tilde\Phi$ is the complementary Young function of $\Phi$ (see Section~\ref{seccion: preliminares}).  

Given $0<p<\infty$, we say that a Young function $\varphi$ has an upper type $p$ if there exists a positive constant $C$ such that $\varphi(st)\leq Cs^p\varphi(t)$, for every $s\geq 1$ and $t\geq0$. If $\varphi$ has an upper type $p$ then has an upper type $q$, for every $q\geq p$. We also say that $\varphi$ has a lower type $p$ if there exists $C>0$ such that the inequality $\varphi(st)\leq Cs^p\varphi(t)$ holds for every $0\leq s\leq 1$ and $t\geq0$. When $\varphi$ has a lower type $p$ it also has a lower type $q$ for every $q\leq p$.

For operators associated to kernels satisfying a regularity of Hörmander type we have the following result.

\begin{teo}\label{teo: F-S para T Hormander}
	Let $\Phi$ be a Young function such that $\tilde\Phi$ has an upper type $r$ and a lower type $s$, for some $1<s<r$. Let $T$ be an operator as in \eqref{eq: representacion integral de T}, with kernel $K\in H_{\Phi}$. Assume that there exist $1<p<r'$ and Young functions $\eta,\varphi$ such that $\eta\in B_{p'}$ and $\eta^{-1}(z)\varphi^{-1}(z)\lesssim \tilde\Phi^{-1}(z)$, for every $z\geq z_0$. If $0\leq u\in L^1_{\rm loc}$ and $v\in \mathrm{RH}_\infty\cap A_q$ with $q=1+(p-1)/r$ then the inequality
	\[uv\left(\left\{x\in \mathbb{R}^n: \frac{|T(fv)(x)|}{v(x)}>t\right\}\right)\leq \frac{C}{t}\int_{\mathbb{R}^n}|f(x)|M_{\varphi_p, v^{1-q'}}u(x)M(\Psi(v))(x)\,dx\]
	holds for every $t>0$, where $\varphi_p(z)=\varphi(z^{1/p})$ and $\Psi(z)=z^{p'+1-q'}\mathcal{X}_{[0,1]}(z)+z^{p'}\mathcal{X}_{[1,\infty)}(z)$.
\end{teo}

We now give an example in order to show that the class of functions satisfying the hypotheses on Theorem~\ref{teo: F-S para T Hormander} is nonempty. Let $r>1$, $1<p<r'$, $\delta\geq 0,$ $0<\varepsilon<\min\{r-1,p'-r\}$ and $\tilde \Phi(t)=t^{r-\varepsilon}(1+\log^+t)^\delta$. Observe that $\Phi\approx \tilde{\tilde \Phi}$, so $\Phi$ is a Young function since it is the complementary of a Young function. We also take $\eta(t)=t^{p'-\tau}$, with $0<\tau<p'-r-\varepsilon$. Then we have that $\tilde\Phi$ has upper type $r$ and a lower type $s$ for every $1<s<r$,  and $\eta\in B_{p'}$. Furthermore,
\[\eta^{-1}(t)\approx t^{1/(p'-\tau)} \quad\textrm{ and }\quad \tilde\Phi^{-1}(t)\approx t^{1/(r'-\varepsilon)(\log t)^{-\delta/(r-\varepsilon)}}\quad \textrm{ for } t\geq e.\]
Therefore, if we take $\varphi(t)=t^q(1+\log^+t)^{\delta q/(r-\varepsilon)}$ where $1/q=1/(r-\varepsilon)-1/(p'-\tau)$ we have the relation $\eta^{-1}(t)\varphi^{-1}(t)\approx \tilde\Phi^{-1}(t)$, for $t\geq e$.

Theorem~\ref{teo: F-S para T Hormander} can be seen as a generalization of \eqref{eq: F-S para T Hormander}, corresponding to $v=1$.

\begin{obs}\label{obs: relacion entre funciones de Young}
	From the hypothesis we have that $\varphi_p(t)\gtrsim \tilde \Phi(t)\geq t$ for $t\geq t_0$. The second inequality is immediate since $\tilde\Phi$ is a Young function. For the first one, given $t\geq t_0$ we can see that $\eta(t)\lesssim t^{p'}$ since $\eta\in B_{p'}$. This implies that $t^{1/p'}\varphi^{-1}(t)\lesssim t$ or equivalently, $\varphi^{-1}(t)\lesssim t^{1/p}$. Then, again by hypothesis
	\[\tilde\Phi^{-1}(t)\gtrsim t^{1/p'}\varphi^{-1}(t)\gtrsim \left(\varphi^{-1}(t)\right)^{p/p'}\varphi^{-1}(t)=\left(\varphi^{-1}(t)\right)^p=\varphi_p^{-1}(t),\]
	which directly implies that $\varphi_p(t)\gtrsim \tilde\Phi(t)$. These relations will be useful in the proof of Theorem~\ref{teo: F-S para T Hormander}.
\end{obs}

The article is organized as follows: in Section~\ref{seccion: preliminares} we give the preliminaries and definitions. Section~\ref{seccion: resultados auxiliares} contains some technical results
that will be useful for the proof of the main theorems  in Section~\ref{seccion: pruebas}.

	\section{Preliminaries and basic definitions}\label{seccion: preliminares}
By a weight $w$ we understand a locally integrable function such that $0<w(x)<\infty$ for almost every $x$. Given $1<p<\infty$, the Muckenhoupt $A_p$ class is defined as the collection of weights $w$ such that the inequality
\[\left(\frac{1}{|Q|}\int_Q w\right)\left(\frac{1}{|Q|}\int_Q w^{1-p'}\right)^{p-1}\leq C\]
holds for some positive constant $C$ and every cube $Q$ in $\mathbb{R}^n$ with sides parallel to the coordinate axes. When necessary, we shall denote by $x_Q$ and $\ell(Q)$ the center and the side-length of the cube $Q$, respectively. 

We say that $w$ belongs to $A_1$ if there exists a positive constant $C$ such that the inequality
\[\frac{1}{|Q|}\int_Q w\leq Cw(x)\]
holds for every cube $Q$ and almost every $x\in Q$.
Finally, for $p=\infty$ the $A_\infty$ class is understood as the collection of all $A_p$ classes, that is, $A_\infty=\bigcup_{p\geq 1}A_p$. 

Given $1\leq p<\infty$, the smallest constant for which the corresponding inequality above holds is denoted by~$[w]_{A_p}$. It is well-known that $A_p$ classes are increasing in $p$, that is, $A_p\subset A_q$ for $p<q$  and that every $w\in A_p$ is doubling, that is, there exists a constant $C>1$ such that
$w(2Q)\leq Cw(Q)$, for every cube $Q$.

For further properties and details about Muckenhoupt classes see, for example, \cite{javi} or \cite{grafakos}.

\medskip

An important property of Muckenhoupt weights is that they satisfy a reverse Hölder condition. Given a real number $s>1$, we say that $w\in \mathrm{RH}_s$ if the inequality
\[\left(\frac{1}{|Q|}\int_Q w^s\right)^{1/s}\leq \frac{C}{|Q|}\int_Q w,\]
holds for some positive constant $C$ and every cube $Q$. The $\rm{RH}_\infty$ class is defined as the set of weights that verify
\[\sup_Q w\leq \frac{C}{|Q|}\int_Q w,\]
for some $C>0$ and every cube $Q$. Given $1<s\leq \infty$, the smallest constant for which the corresponding inequality above holds is denoted by $[w]_{\mathrm{RH}_s}$. It is well-known that reverse Hölder classes are decreasing on $s$, that is, ${\rm{RH}_\infty}\subset\mathrm{RH}_s\subset \mathrm{RH}_t$ for every $1<t<s$. 

The next lemma establishes some useful properties of $\rm{RH}_\infty$ weights that we shall use later. A proof can be found in \cite{Cruz-Uribe-Neugebauer}.

\begin{lema}\label{lema: potencia negativa de RHinf en A1 y positivas en RHinf}
	Let $w$ be a weight.
	\begin{enumerate}[\rm (a)]
		\item \label{item a - lema: potencia negativa de RHinf en A1 y positivas en RHinf}If $w\in\mathrm{RH}_\infty\cap A_p$, then $w^{1-p'}\in A_1$;
		\item \label{item b - lema: potencia negativa de RHinf en A1 y positivas en RHinf}if $w\in \mathrm{RH}_\infty$, then $w^r\in \mathrm{RH}_\infty$ for every $r>0$;
		\item \label{item c - lema: potencia negativa de RHinf en A1 y positivas en RHinf}if $w\in A_1$, then $w^{-1}\in \rm{RH}_\infty$.
	\end{enumerate}
\end{lema}

\medskip

We say that $\varphi:[0,\infty)\to[0,\infty)$ is a Young function if it is convex, strictly increasing and also satisfies $\varphi(0)=0$ and $\lim_{t\to\infty}\varphi(t)=\infty$. The generalized inverse $\varphi^{-1}$ of $\varphi$ is defined by 
\[\varphi^{-1}(t)=\inf\{s\geq 0: \varphi(s)\geq t\},\]
where we understand $\inf\emptyset =\infty$. When $\varphi$ is a Young function that verifies $0<\varphi(t)<\infty$ for every $t>0$ it can be seen that $\varphi$ is invertible and the generalized inverse of $\varphi$ is its actual inverse function. Throughout this paper we shall deal with this type of Young functions.

The complementary function of the Young function $\varphi$ is denoted by $\tilde\varphi$ and defined for $t\geq 0$ by
\[\tilde\varphi (t)=\sup\{ts-\varphi(s):s\geq 0\}.\]
It is well-known that $\tilde\varphi$ is also a Young function and further
\begin{equation}\label{eq: preliminares - producto de inversas como t}
\varphi^{-1}(t)\tilde\varphi^{-1}(t)\approx t.
\end{equation}
Given a Young function $\varphi$ and a Muckenhoupt weight $w$, the generalized maximal operator $M_{\varphi, w}$ is defined, for $f$ such that $\varphi(f)\in L^1_{\rm loc}$, by
\[M_{\varphi, w}f(x)=\sup_{Q\ni x}\|f\|_{\varphi,Q,w},\]
where $\|f\|_{\varphi,Q,w}$ is an average of Luxemburg type given by the expression
\[\|f\|_{\varphi,Q,w}=\inf\left\{\lambda>0: \frac{1}{w(Q)}\int_Q \varphi\left(\frac{|f(y)|}{\lambda}\right)w(y)\,dy\leq 1\right\},\]
and this infimum is actually a minimum, since it is easy to see that
\[\frac{1}{w(Q)}\int_Q \varphi\left(\frac{|f(y)|}{\|f\|_{\varphi,Q,w}}\right)w(y)\,dy\leq 1.\]
When $w=1$ we simply write $\|f\|_{\varphi,Q}$ and $M_{\varphi,w}=M_{\varphi}$. When $\varphi(t)=t$, the operator $M_{\varphi,w}$ is just the classical Hardy-Littlewood maximal function with respect to the measure $d\mu(x)=w(x)\,dx$.

 Notice that when $\varphi$ is a Young function which has both a lower type $r_1$ and an upper type $r_2$ with $1<r_1<r_2$, we have $M_{r_1}\lesssim M_\varphi\lesssim M_{r_2}$. 

If $\Phi,\Psi$ and $\varphi$ are Young functions satisfying 
\[\Phi^{-1}(t)\Psi^{-1}(t)\lesssim \varphi^{-1}(t)\]
for $t\geq t_0\geq 0$, then 
\begin{equation}\label{eq: preliminares - relacion entre inversas para Holder}
\varphi(st)\lesssim \Phi(s)+\Psi(t),
\end{equation}
for every $s,t\geq 0$. As a consequence of this estimate we obtain the generalized Hölder inequality
\[\|fg\|_{\varphi,E,w}\lesssim \|f\|_{\Phi,E,w}\|g\|_{\Psi,E,w},\]
for every doubling weight $w$ and every measurable set $E$ such that $|E|<\infty$. Particularly, in views of \eqref{eq: preliminares - producto de inversas como t} we get that
\begin{equation}\label{eq: preliminares - Holder generalizada}
\frac{1}{w(E)}\int_E |fg|w\lesssim \|f\|_{\varphi,E,w}\|g\|_{\tilde\varphi,E,w}.
\end{equation}

We say that a Young function $\varphi$ belongs to $B_p$, $p>1$, if there exists a positive constant $c$ such that
\[\int_c^\infty \frac{\varphi(t)}{t^p}\,\frac{dt}{t}<\infty.\]
These classes were introduced in \cite{Perez-95-Onsuf} and play a fundamental role in the Fefferman-Stein estimates for CZO.

\section{Auxiliary results}\label{seccion: resultados auxiliares}

We devote this section to state and prove some results that will be useful in the  proof of our main result.

The theorem below establishes a strong type Fefferman-Stein estimate for CZO. 

\begin{teo}[\cite{P94}]\label{teo: tipo fuerte de T con peso arbitrario}
	Let $T$ be a CZO, $1<p<\infty$ and $\varphi$ a Young function that verifies $\varphi\in B_p$. Then there exists a positive constant $C$ such that for every weight $w$ we have that
	\[\int_{\mathbb{R}^n}|Tf(x)|^pw(x)\,dx\leq C\int_{\mathbb{R}^n}|f(x)|^pM_\varphi w(x)\,dx.\]
\end{teo}

The following result states a Coifman type estimate for operators associated to kernels with less regularity.

\begin{teo}[\cite{Lorente-Riveros-delaTorre05}]\label{teo: tipo Coifman T Hormander}
	Let $\Phi$ be a Young function and $T$ as in \eqref{eq: representacion integral de T}, with kernel $K\in H_\Phi$. Then for every $0<p<\infty$ and $w\in A_\infty$ the inequality
	\[\int_{\mathbb{R}^n}|Tf(x)|^pw(x)\,dx\leq C\int_{\mathbb{R}^n}\left(M_{\tilde\Phi}f(x)\right)^pw(x)\,dx\]
	holds for every $f$ such that the left hand-side is finite.
\end{teo}

The following lemma will be an important tool in order to obtain the main result. 

\begin{lema}\label{lema: comparacion peso por caracteristica}
	Let $\varphi$ be a Young function, $w$ a doubling weight, $f$ such that $M_{\varphi,w} f (x)<\infty$ almost everywhere and $Q$ be a fixed cube. Then
	\[M_{\varphi,w}(f\mathcal{X}_{\mathbb{R}^n\backslash RQ})(x)\approx M_{\varphi,w}(f\mathcal{X}_{\mathbb{R}^n\backslash RQ})(y)\]
	for every $x,y\in Q$, where $R=4\sqrt{n}$. 
\end{lema}

\begin{proof}
	Fix $x$ and $y$ in $Q$ and let $Q'$ be a cube containing $x$. We can assume that $Q'\cap \mathbb{R}^n\backslash RQ\neq\emptyset$, since $\|f\|_{\varphi,Q',w}=0$ otherwise. We shall prove that 
	\begin{equation}\label{eq: lema: comparacion peso por caracteristica - eq1}
	\ell(Q')\geq \frac34 \ell(Q).
	\end{equation} 
	Indeed, let $B_{Q'}=B(x_{Q'}, \ell(Q')/2)$ and $B_Q=B(x_Q,\ell(Q)\sqrt{n}/2)$. Observe that $Q'\cap \mathbb{R}^n\backslash RQ\neq\emptyset$ implies that $B_{Q'}\cap \mathbb{R}^n\backslash RQ\neq\emptyset$. If \eqref{eq: lema: comparacion peso por caracteristica - eq1} does not hold, for $z\in B_{Q'}$ we would have
	\begin{align*}
	|z-x_Q|&\leq |z-x|+|x-x_Q|\\
	&\leq \sqrt{n}\ell(Q')+\frac{\sqrt{n}}{2}\ell(Q)\\
		   &<\left(\frac{3\sqrt{n}}{4}+\frac{\sqrt{n}}{2}\right)\ell(Q) \\
		   &<\frac{R}{2}\ell(Q),
	\end{align*}
		which yields $B_{Q'}\subseteq B(x_Q, R\ell(Q)/2)\subseteq RQ$, a contradiction. Therefore \eqref{eq: lema: comparacion peso por caracteristica - eq1} holds. Then we have that $Q\subseteq RQ'$. Indeed, if $z\in Q$ we get
		\begin{align*}
	|z-x_{Q'}|&\leq |z-x|+|x-x_{Q'}|\\
	&\leq \sqrt{n}\ell(Q)+\frac{\sqrt{n}}{2}\ell(Q')\\
	&\leq \left(\frac{4\sqrt{n}}{3}+\frac{\sqrt{n}}{2}\right)\ell(Q')\\
	&<2\sqrt{n}\ell(Q')=\frac{R}{2}\ell(Q'),
	\end{align*}
	which implies that $Q\subseteq B(x_{Q'},\ell(RQ')/2)\subseteq RQ'$. Thus
\[\frac{1}{w(Q')}\int_{Q'}\varphi\left(\frac{|f|}{\|f\|_{\varphi, RQ',w}}\right)w\leq \frac{w(RQ')}{w(Q')}\frac{1}{w(RQ')}\int_{RQ'}\varphi\left(\frac{|f|}{\|f\|_{\varphi, RQ',w}}\right)w\\
\leq C,\]
	since $w$ is doubling. This yields $\|f\|_{\varphi,Q',w}\leq C\|f\|_{\varphi,RQ',w}\leq CM_{\varphi, w}f(y)$, for every $Q'$ containing~$x$, which finally implies that $M_{\varphi,w}f(x)\leq CM_{\varphi,w}f(y)$. The other inequality can be achieved analogously by interchanging the roles of $x$ and $y$.
\end{proof}

The following result gives a relation between $M_{\varphi,w}$ and the unweighted version $M_\varphi$, when $w$ belongs to the $A_1$ class.

\begin{lema}\label{lema: relacion maximal generalizada con peso A1}
	Let $w\in A_1$ and $\varphi$ be a Young function.
	\begin{enumerate}[\rm (a)]
		\item \label{item a - lema: relacion maximal generalizada con peso A1}  There exists a positive constant $C$ such that
		\[M_{\varphi}f(x)\leq C M_{\varphi,w}f(x),\]
		for every $f$ such that $M_{\varphi,w}f(x)<\infty$ a.e.;
		\item \label{item b - lema: relacion maximal generalizada con peso A1} if $w^r\in A_1$ for some $r>1$, then
		\[M_{\varphi,w}f(x)\leq C M_{\varphi,w^r}f(x),\]
		for every $f$ such that $M_{\varphi,w^r}f(x)<\infty$ a.e.
		\end{enumerate}
\end{lema}

\begin{proof}
	For \eqref{item a - lema: potencia negativa de RHinf en A1 y positivas en RHinf}, fix $x$ and a cube $Q\ni x$.  Since $w\in A_1$ we have that 
	\begin{align*}
	\frac{1}{|Q|}\int_Q \varphi\left(\frac{|f|}{\lambda}\right)&=\frac{w(Q)}{|Q|}\frac{1}{w(Q)}\int_Q \varphi\left(\frac{|f|}{\lambda}\right)ww^{-1}\\
	&\leq \left(\sup_Q w^{-1}\right)[w]_{A_1}\left(\inf_Q w\right)\frac{1}{w(Q)}\int_Q \varphi\left(\frac{|f|}{\lambda}\right)w\\
	&\leq [w]_{A_1},
	\end{align*}
	if we take $\lambda=\|f\|_{\varphi,Q,w}$. Then we have $\|f\|_{\varphi,Q}\leq [w]_{A_1}\|f\|_{\varphi,Q,w}\leq [w]_{A_1} M_{\varphi,w}f(x)$, for every cube $Q$ that contains $x$. By taking supremum on these $Q$ we obtain the desired inequality.
	
	The proof of \eqref{item b - lema: relacion maximal generalizada con peso A1} follows similar lines. Indeed, by Lemma~\ref{lema: potencia negativa de RHinf en A1 y positivas en RHinf} we have that $w^{1-r}\in\rm{RH}_\infty$, so
	\begin{align*}
	\frac{1}{w(Q)}\int_Q \varphi\left(\frac{|f|}{\lambda}\right)w&=\frac{w^r(Q)}{w(Q)}\frac{1}{w^r(Q)}\int_Q \varphi\left(\frac{|f|}{\lambda}\right)w^rw^{1-r}\\
	&\leq \left(\sup_Q w^{1-r}\right)[w^r]_{A_1}\left(\inf_Q w^r\right)\frac{|Q|}{w(Q)}\frac{1}{w^r(Q)}\int_Q \varphi\left(\frac{|f|}{\lambda}\right)w^r\\
	&\leq \left[w^r\right]_{A_1}\left[w^{1-r}\right]_{\rm{RH}_\infty},
	\end{align*}
	provided we choose $\lambda=\|f\|_{\varphi,Q,w^r}$. \qedhere
\end{proof}

The next lemma gives a bound for functions of $L\log L$ type that we shall need in the main result. A proof can be found in \cite{Berra}.

\begin{lema}\label{lema: acotacion LlogL por potencia}
Let $\delta\geq 0$ and $\varphi(t)=t(1+\log^+t)^\delta$. For every $\varepsilon>0$ there exists a positive constant $C=C(\varepsilon,\delta)$ such that
\[\varphi(t)\leq Ct^{1+\varepsilon}, \quad\textrm{ for }\quad t\geq 1.\]
Moreover, the constant $C$ can be taken as $C=\max\left\{1, (\delta/\varepsilon)^\delta\right\}.$	
\end{lema}

\section{Proof of the main results}\label{seccion: pruebas}

\begin{proof}[Proof of Theorem~\ref{teo: acotacion mixta general para T}]
	Let us first assume that $u$ is bounded. We fix $t>0$ and perform the Calder\'on-Zygmund decomposition of $f$ at level $t$ with respect to the measure $d\mu(x)=v(x)\,dx$. Let us observe that  $v\in \mathrm{RH}_\infty$, so that  $v\in A_\infty$ and therefore $\mu$ is a doubling measure. We obtain a collection of disjoint dyadic cubes $\{Q_j\}_{j=1}^\infty$ satisfying $t<f_{Q_j}^v\leq Ct$, where $f_{Q_j}^v$ is given by
	\[\frac{1}{v(Q_j)}\int_{Q_j}f(y)v(y)\,dy.\]
	If we write $\Omega=\bigcup_{j=1}^{\infty}Q_j$, then we have that $f(x)\leq t$ for almost every $x\in{\mathbb{R}}^n\backslash\Omega$. We also decompose $f$ as $f=g+h$, where
	\begin{equation*}
	g(x)=\left\{
	\begin{array}{ccl}
	f(x),& \textrm{ if } &x\in\mathbb{R}^n\backslash\Omega;\\
	f_{Q_j}^v,&\textrm{ if }& x\in Q_j,
	\end{array}
	\right.
	\end{equation*}
	and $h(x)=\sum_{j=0}^{\infty}{h_j(x)}$, with
	\begin{equation*}
	h_j(x)=\left(f(x)-f_{Q_j}^v\right)\mathcal{X}_{Q_j}(x).
	\end{equation*}
	It follows that $g(x)\leq Ct$ almost everywhere, every
	$h_j$ is supported on $Q_j$ and
	\begin{equation}\label{eq: hj integra cero contra v}
	\int_{Q_j}h_j(y)v(y)\,dy=0.
	\end{equation}
	Let $Q_j^*=RQ_j$, where $R=4\sqrt{n}$ as in Lemma~\ref{lema: comparacion peso por caracteristica}  and $\Omega^*=\bigcup_j Q_j^*$. We proceed as follows
	\begin{align*}
	uv\left(\left\{x\in \mathbb{R}^n: \left|\frac{T(fv)}{v}\right|>t\right\}\right)&\leq uv\left(\left\{x\in \mathbb{R}^n\backslash \Omega^*: \left|\frac{T(gv)}{v}\right|>\frac{t}{2}\right\}\right) + uv(\Omega^*)\\
	&+uv\left(\left\{x\in \mathbb{R}^n\backslash \Omega^*: \left|\frac{T(hv)}{v}\right|>\frac{t}{2}\right\}\right)\\
	&= I+II+III.
	\end{align*}
	We shall estimate each term separately. For $I$, let us fix $p>\max\{q,1+1/\delta\}$. Then we have that $p'<1+\delta$ and  
	\begin{align*}
	\int_e^\infty\left(\frac{t}{\varphi(t)}\right)^{p-1}\,\frac{dt}{t}&=\int_e^\infty\left(\frac{1}{\log t}\right)^{\delta(p-1)}\,\frac{dt}{t}\\
	&=\int_1^\infty y^{(1-p)\delta}\,dy\\
	&<\infty,
	\end{align*}
	since $\delta(p-1)>1$ by the choice of $p$. If we set $u^*=u\mathcal{X}_{\mathbb{R}^n\backslash \Omega^*}$, by applying Tchebychev inequality and Theorem~\ref{teo: tipo fuerte de T con peso arbitrario} we obtain 
	\begin{align*}
	I&\leq\frac{C}{t^{p'}}\int_{\mathbb{R}^n} |T(gv)|^{p'}uv^{1-p'}\mathcal{X}_{\mathbb{R}^n\backslash \Omega^*}\\
	&=\frac{C}{t^{p'}}\int_{\mathbb{R}^n} |T(gv)|^{p'}u^*v^{1-p'}\\
	&\leq\frac{C}{t^{p'}}\int_{\mathbb{R}^n} |gv|^{p'}M_\varphi\left(u^*v^{1-p'}\right).
	\end{align*}
	
	Let us estimate $M_\varphi\left(u^*v^{1-p'}\right)$. Recall that we have $v\in \mathrm{RH}_\infty\cap A_p$ since $p>q$, so by item~\eqref{item a - lema: potencia negativa de RHinf en A1 y positivas en RHinf} of Lemma~\ref{lema: potencia negativa de RHinf en A1 y positivas en RHinf} we get $v^{1-p'}\in A_1$. We shall prove that there exists a positive constant $C$ verifying 
	\begin{equation}\label{eq: teo acotacion mixta general para T - relacion maximales con y sin peso}
	M_\varphi\left(u^*v^{1-p'}\right)(x)\leq CM_{\varphi,v^{1-q'}}(u^*)(x)v^{-p'}(x)\Psi(v(x)) \quad\textrm{ for a.e. }x. 
	\end{equation}
	Fix $x$ and $Q$ a cube containing $x$. By taking $\lambda=\|u^*\|_{\varphi, Q, v^{1-q'}}$ we have that
	\begin{align*}
	\frac{1}{|Q|}\int_Q \varphi\left(\frac{u^*v^{1-p'}}{\lambda}\right)&=\frac{1}{|Q|}\int_{Q\cap\{v^{1-p'}\leq e\}} \varphi\left(\frac{u^*v^{1-p'}}{\lambda}\right)+\frac{1}{|Q|}\int_{Q\cap\{v^{1-p'}>e\}} \varphi\left(\frac{u^*v^{1-p'}}{\lambda}\right)\\
	&=I_1+I_2.
	\end{align*}
	By using that $\varphi$ is submultiplicative and Lemma~\ref{lema: relacion maximal generalizada con peso A1},  for $I_1$ we have that
	\[I_1\leq \frac{C}{|Q|}\int_Q\varphi\left(\frac{u^*}{\|u^*\|_{\varphi, Q , v^{1-q'}}}\right)\leq \frac{C}{|Q|}\int_Q\varphi\left(\frac{u^*}{\|u^*\|_{\varphi, Q}}\right)\leq C.\]
	 In order to estimate $I_2$, let $\varepsilon=(q'-1)/(p'-1)-1$. Observe that $\varepsilon>0$ since $p'<q'$. By applying Lemma~\ref{lema: acotacion LlogL por potencia}  we get that  
	 \begin{align*}
	 	I_2&\leq \frac{C}{|Q|}\int_Q \varphi\left(\frac{u^*(y)}{\lambda}\right)v^{(1-p')(1+\varepsilon)}\,dy\\
	&= \frac{C}{|Q|}\int_Q \varphi\left(\frac{u^*(y)}{\lambda}\right)v^{1-q'}(y)\,dy\\
	&\leq C\frac{v^{1-q'}(Q)}{|Q|}\left(\frac{1}{v^{1-q'}(Q)}\int_Q\varphi\left(\frac{u^*(y)}{\|u^*\|_{\varphi,Q,v^{1-q'}}}\right)v^{1-q'}(y)\,dy\right)\\
	&\leq C\left[v^{1-q'}\right]_{A_1}\max\left\{1,v^{1-q'}(x)\right\},
	 \end{align*}
	 since $v^{1-q'}\in A_1$. Therefore, we can conclude that
	 \begin{align*}
	 \|u^*v^{1-p'}\|_{\varphi,Q}&\leq C\lambda\max\left\{1,v^{1-q'}(x)\right\}\\
	 &= C\max\left\{1,v^{1-q'}(x)\right\}\|u^*\|_{\varphi,Q,v^{1-q'}}\\
	 &\leq C H(x)M_{\varphi,v^{1-q'}}u^*(x),
	 \end{align*}
	 for every cube $Q$ that contains $x$ and where $H(x)=\max\left\{1,v^{1-q'}(x)\right\}$. This finally yields
	 \[M_\varphi\left(u^*v^{1-p'}\right)(x)\leq C H(x)M_{\varphi,v^{1-q'}}u^*(x).\]
	 To obtain \eqref{eq: teo acotacion mixta general para T - relacion maximales con y sin peso} it only remains to show that $H(x)v^{p'}(x)\leq \Psi(v(x))$. This can be achieved by noting that when $v\leq 1$ we have  $v^{p'}H=v^{p'+1-q'}$, and we get $v^{p'}H=v^{p'}$ otherwise.
	 
	We now return to the estimate of $I$. We have that \label{pag: estimacion de I}
	\begin{align*}
	I&\leq \frac{C}{t^{p'}}\int_{\mathbb{R}^n}|gv|^{p'}\left(M_{\varphi, v^{1-q'}}u^*\right)v^{-p'}\Psi(v)\\
	&\leq \frac{C}{t}\int_{\mathbb{R}^n}|g|\left(M_{\varphi, v^{1-q'}}u^*\right)\Psi(v)\\
	&\leq \frac{C}{t}\int_{\mathbb{R}^n\backslash \Omega}|f|\left(M_{\varphi, v^{1-q'}}u\right)M(\Psi(v))+\frac{C}{t}\int_{\Omega}|f_{Q_j}^v|\left(M_{\varphi, v^{1-q'}}u^*\right)\Psi(v).
	\end{align*}
	Let $u_j^*=u\mathcal{X}_{\mathbb{R}^n\backslash RQ_j}$. For the integral over $\Omega$ we shall use Lemma~\ref{lema: comparacion peso por caracteristica} to obtain
	\begin{align*}
	\frac{C}{t}\int_{\Omega}|f_{Q_j}^v|\left(M_{\varphi, v^{1-q'}}u^*\right)\Psi(v)&\leq\frac{C}{t}\sum_j\int_{Q_j}|f_{Q_j}^v|\left(M_{\varphi, v^{1-q'}}u_j^*\right)\Psi(v)\\
	&\leq \frac{C}{t}\sum_j\inf_{Q_j} \left(M_{\varphi, v^{1-q'}}u_j^*\right)\frac{(\Psi\circ v)(Q_j)}{v(Q_j)}\int_{Q_j}|f|v\\
	&\leq \frac{C}{t}[v]_{\mathrm{RH}_\infty}\left(\inf_{Q_j} M_{\varphi, v^{1-q'}}u_j^*\right)\frac{(\Psi\circ v)(Q_j)}{|Q_j|}\int_{Q_j}|f|\\
	&\leq \frac{C}{t}[v]_{\mathrm{RH}_\infty}\sum_j\int_{Q_j}|f|\left(M_{\varphi, v^{1-q'}}u^*\right)M(\Psi(v))\\
	&\leq \frac{C}{t}[v]_{\mathrm{RH}_\infty}\int_{\Omega}|f|\left(M_{\varphi, v^{1-q'}}u\right)\,M(\Psi(v)),
	\end{align*}
	so we achieved the desired estimate for $I$. 
	
	For $II$, by virtue of  Lemma~\ref{lema: potencia negativa de RHinf en A1 y positivas en RHinf} and the fact that $v$ is doubling we have
	\begin{align*}
	uv(Q_j^*)&\leq v^{1-q'}(Q_j^*)\|u\|_{\varphi,Q_j^*,v^{1-q'}}\left[\frac{1}{v^{1-q'}(Q_j^*)}\int_{Q_j^*}\varphi\left(\frac{u}{\|u\|_{\varphi,Q_j^*,v^{1-q'}}}\right)v^{1-q'}\right]\left(\sup_{Q_j^*} v^{q'}\right)\\
	&\leq \left[v^{q'}\right]_{\rm{RH}_\infty}\frac{v^{1-q'}(Q_j^*)}{|Q_j^*|}v^{q'}(Q_j^*)\|u\|_{\varphi,Q_j^*,v^{1-q'}}\\
	&\leq C\left[v^{q'}\right]_{\rm{RH}_\infty}\left[v^{1-q'}\right]_{A_1}v(Q_j)\|u\|_{\varphi,Q_j^*,v^{1-q'}}\\
	&\leq \frac{C}{t}\int_{Q_j}|f|v\left(M_{\varphi,v^{1-q'}}u\right)\\
	&\leq \frac{C}{t}\int_{Q_j}|f|\left(M_{\varphi,v^{1-q'}}u\right)M(\Psi(v)),
	\end{align*}
	where in the last inequality we have used  that $\Psi(s)\geq s$.
	Therefore,\label{pag: estimacion de II}
	\begin{align*}
	uv(\Omega^*)&=\sum_j uv(Q_j^*)\\
	&\leq \frac{C}{t}\sum_j \int_{Q_j}|f|\left(M_{\varphi,v^{1-q'}}u\right)M(\Psi(v))\\
	& \leq\frac{C}{t} \int_{\mathbb{R}^n}|f|\left(M_{\varphi,v^{1-q'}}u\right)M(\Psi(v)).
	\end{align*}
	It only remains to estimate $III$. We denote $A_{j,k}=\{x: 2^{k-1}r_j<|x-x_{Q_j}|\leq 2^{k}r_j\}$, where $r_j=R\ell(Q_j)/2$ and use the integral representation of $T$ given by \eqref{eq: representacion integral de T}. By combining \eqref{eq: hj integra cero contra v} with the smoothness condition \eqref{eq:prop del nucleo} on $K$  we get
\begin{align*}
III&\leq uv\left(\left\{x\in \mathbb{R}^n\backslash\Omega^*: \sum_j\frac{|T(h_jv)|}{v}>\frac{t}{2}\right\}\right)\\
&\leq \frac{C}{t}\int_{\mathbb{R}^n\backslash\Omega^*}\sum_j|T(h_jv)(x)|u(x)\,dx\\
&\leq \frac{C}{t}\sum_j\int_{\mathbb{R}^n\backslash Q_j^*}\left|\int_{Q_j}h_j(y)v(y)(K(x-y)-K(x-x_{Q_j}))\,dy\right|u(x)\,dx\\
&\leq \frac{C}{t}\sum_j\int_{Q_j}|h_j(y)|v(y)\int_{\mathbb{R}^n\backslash Q_j^*}|K(x-y)-K(x-x_{Q_j})|u_j^*(x)\,dx\,dy\\
&\leq \frac{C}{t}\sum_j\int_{Q_j}|h_j(y)|v(y)\sum_{k=1}^{\infty}{\int_{A_{j,k}}}|K(x-y)-K(x-x_{Q_j})|u_j^*(x)\,dx\,dy\\
&= \frac{C}{t}\sum_j\int_{Q_j}|h_j(y)|v(y)\sum_{k=1}^{\infty}{\int_{A_{j,k}}}\frac{|y-x_{Q_j}|}{|x-x_{Q_j}|^{n+1}}u_j^*(x)\,dx\,dy.
\end{align*}
For every fixed $y\in Q_j$ we have that
\begin{align*}
\sum_{k=1}^{\infty}{\int_{A_{j,k}}}\frac{|y-x_{Q_j}|}{|x-x_{Q_j}|^{n+1}}u_j^*(x)\,dx&\leq C\sum_{k=1}^{\infty}\frac{\sqrt{n}\ell(Q_j)}{2r_j}\frac{2^{-k}}{(2^kr_j)^n}\int_{B\left(x_{Q_j},2^kr_j\right)}u_j^*(x)\,dx\\
&\leq CMu_j^*(y)\sum_{k=1}^\infty 2^{-k}\\
&\leq CMu_j^*(y).
\end{align*}
Therefore, by Lemma~\ref{lema: comparacion peso por caracteristica} we obtain
\begin{align*}
III&\leq \frac{C}{t}\sum_j\int_{Q_j}|f|v\left(\inf_{Q_j}Mu_j^*\right)+C\sum_j\int_{Q_j}|f_{Q_j}^v|v\left(\inf_{Q_j}Mu_j^*\right)\\
&=A+B.
\end{align*}
Applying Lemma~\ref{lema: relacion maximal generalizada con peso A1} we have that $Mu\leq M_\varphi u\leq CM_{\varphi,v^{1-q'}}u$ and this yields\label{pag: estimacion de III}
\[A\leq \frac{C}{t}\int_{\Omega} |f|\left(M_{\varphi,v^{1-q'}}u\right)M(\Psi(v)).\]
On the other hand,
\begin{align*}
B&\leq \frac{C}{t}\sum_j \int_{Q_j}|f|v\left(\inf_{Q_j}M_j^*u\right)\\
&\leq \frac{C}{t}\sum_j \int_{Q_j}|f|vM_\varphi u\\
&\leq \frac{C}{t}\int_{\mathbb{R}^n}|f|\left(M_{\varphi,v^{1-q'}}u\right) M(\Psi(v)).
\end{align*}	
This completes the proof when $u$ is bounded, with a constant $C$ that does not depend on $u$. For the general case, given $u$ we set $u_N(x)=\min\{u(x),N\}$ for every $N\in\mathbb{N}$. Then we have that
\begin{align*}
u_Nv\left(\left\{x\in \mathbb{R}^n: \frac{|T(fv)(x)|}{v(x)}>t\right\}\right)&\leq \frac{C}{t}\int_{\mathbb{R}^n}|f|\left(M_{\varphi, v^{1-q'}}u_N\right)M(\Psi(v))\\
&\leq\frac{C}{t}\int_{\mathbb{R}^n}|f(x)|\left(M_{\varphi, v^{1-q'}}u\right)M(\Psi(v))
\end{align*}
for every $N\in\mathbb{N}$ and with a positive constant $C$ that does not depend on $N$. Since $u_N\nearrow u$, the monotone convergence theorem allows us to show that the estimate for $u$ also holds.\qedhere
\end{proof}

\medskip

\begin{proof}[Proof of Theorem~\ref{teo: F-S para T Hormander}]
	We shall first consider the case $u$ bounded. Fix $t>0$ and, as in the proof of Theorem~\ref{teo: acotacion mixta general para T}, perform the Calder\'on-Zygmund decomposition of $f$ at level $t$ with respect to the measure $d\mu(x)=v(x)\,dx$. Therefore we obtain a collection of disjoint dyadic cubes $\{Q_j\}_{j=1}^\infty$, $\Omega$, $g$ and $h$ as in that proof. We take $Q_j^*=cRQ_j$, where $R$ is the dimensional constant given by Lemma~\ref{lema: comparacion peso por caracteristica} and $c\geq 1$ is the constant appearing on the $L^{\Phi}-$Hörmander condition for $K$. By using the same notation as in Theorem~\ref{teo: acotacion mixta general para T} we get
	\begin{align*}
	uv\left(\left\{x\in \mathbb{R}^n: \left|\frac{T(fv)}{v}\right|>t\right\}\right)&\leq uv\left(\left\{x\in \mathbb{R}^n\backslash \Omega^*: \left|\frac{T(gv)}{v}\right|>\frac{t}{2}\right\}\right) + uv(\Omega^*)\\
	&+uv\left(\left\{x\in \mathbb{R}^n\backslash \Omega^*: \left|\frac{T(hv)}{v}\right|>\frac{t}{2}\right\}\right)\\
	&= I+II+III.
	\end{align*}
	Since $\tilde \Phi$ has a lower type $s$, we have $M_s\lesssim M_{\tilde\Phi}$. Recall that $p'>r$ since $p<r'$. 
	By using the fact that $M_sg$ is an $A_1$ weight for every measurable and nonnegative function $g$ such that $M_s g$ is finite almost everywhere, we apply Tchebychev inequality with $p'$ and Theorem~\ref{teo: tipo Coifman T Hormander} in order to get	
	\begin{align*}
	I&\leq\frac{C}{t^{p'}}\int_{\mathbb{R}^n} |T(gv)|^{p'}uv^{1-p'}\mathcal{X}_{\mathbb{R}^n\backslash \Omega^*}\\
	&\leq\frac{C}{t^{p'}}\int_{\mathbb{R}^n} |T(gv)|^{p'}M_s\left(u^*v^{1-p'}\right)\\
	&\leq\frac{C}{t^{p'}}\int_{\mathbb{R}^n} \left[M_{r}(gv)\right]^{p'}M_s\left(u^*v^{1-p'}\right)\\
	&\leq \frac{C}{t^{p'}}\int_{\mathbb{R}^n} (|g|v)^{p'}M_s\left(u^*v^{1-p'}\right)\\
	&\leq \frac{C}{t^{p'}}\int_{\mathbb{R}^n} (|g|v)^{p'}M_{\tilde\Phi}\left(u^*v^{1-p'}\right).
	\end{align*}
	Notice that we could apply Theorem~\ref{teo: tipo Coifman T Hormander} because $\|T(gv)\|_{L^{p'}(w)}<\infty$, where $w=M_s(u^*v^{1-p'})\in~A_1$. Indeed, if we first assume $w\in A_1\cap L^\infty$, we get
	\[\int_{\mathbb{R}^n} |T(gv)|^{p'}w\leq \|w\|_{L^\infty}\int_{\mathbb{R}^n} |T(gv)|^{p'}\leq C\|w\|_{L^\infty}\int_{\mathbb{R}^n}(|g|v)^{p'}<\infty,\]
	since $f$ is bounded with compact support and $T$ is bounded in $L^{p'}$ because $K\in H_{\Phi}\subset H_1$ (see, for example, \cite{javi}). For the general case, we can take $w_N=\min\{w,N\}$ for every $N\in\mathbb{N}$. Then every $w_N$ belongs to $A_1$ and $[w_N]_{A_1}\leq [w]_{A_1}$. This allows us to deduce the inequality in Theorem~\ref{teo: tipo Coifman T Hormander} for $w_N$ and $C$ independent of $N$. By letting $N\to\infty$ we are done. 
		
	We proceed now to estimate $M_{\tilde\Phi}(u^*v^{1-p'})(x)$. We shall prove that 
	\begin{equation}\label{eq: teo: F-S para T Hormander - relacion maximales con y sin peso}
	M_{\tilde\Phi}\left(u^*v^{1-p'}\right)(x)\leq C\left(M_{\varphi_p,v^{1-q'}}u^*\right)(x)v^{-p'}(x)\Psi(v(x)) \quad\textrm{ for a.e. }x. 
	\end{equation}
	Fix $x$ and $Q$ a cube containing $x$. By hypothesis and Lemma~\ref{lema: potencia negativa de RHinf en A1 y positivas en RHinf}, we have that $v^{1-q'}$ is an $A_1$ weight. By taking  $\lambda=\|u^*\|_{\varphi_p, Q, v^{1-q'}}$, we have that 
	\[\frac{1}{|Q|}\int_Q \tilde\Phi\left(\frac{u^*v^{1-p'}}{\lambda}\right)=\frac{1}{|Q|}\int_{Q\cap\{v^{1-p'}\leq 1\}}+\frac{1}{|Q|}\int_{Q\cap\{v^{1-p'}> 1\}}=A+B.\]
	It is easy to see that $A$ is bounded by a constant $C$, since $\tilde \Phi(z) \lesssim \varphi_p(z)$ for large $z$ and $\|u^*\|_{\varphi_p,Q}\leq \lambda$. In order to estimate $B$, we shall apply the upper type of $\tilde\Phi$ combined with \eqref{eq: preliminares - relacion entre inversas para Holder}, the fact that $\eta(t)\leq Ct^{p'}$ and  $t\lesssim \varphi_p(t)$ (see Remark~\ref{obs: relacion entre funciones de Young}) to get
	\begin{align*}
	B&\leq \frac{C}{|Q|}\int_{Q\cap\{v^{1-p'}>1\}} \tilde\Phi\left(\frac{u^*}{\lambda}\right)v^{r(1-p')}\\
	&=\frac{C}{|Q|}\int_Q \tilde\Phi\left(\left(\frac{u^*}{\lambda}\right)^{1/p}\left(\frac{u^*}{\lambda}\right)^{1/p'}\right)v^{1-q'}\\
	&\leq \frac{C}{|Q|}\int_Q\varphi_p\left(\frac{u^*}{\lambda}\right)v^{1-q'}+\frac{C}{|Q|}\int_Q\frac{u^*}{\lambda}v^{1-q'}\\
	&\leq \frac{1}{|Q|}\int_Q\varphi_p\left(\frac{u^*}{\|u^*\|_{\varphi_p,Q,v^{1-q'}}}\right)v^{1-q'}+C\left(\frac{1}{|Q|}\int_Q\frac{u^*}{\|u^*\|_{\varphi_p,Q,v^{1-q'}}}v^{1-q'}\right)\\
	&\leq C\frac{v^{1-q'}(Q)}{|Q|}\\
	&\leq C\left[v^{1-q'}\right]_{A_1}\max\{1,v^{1-q'}(x)\}.
	\end{align*}
	 By Lemma~\ref{lema: relacion maximal generalizada con peso A1} we have that
	\[\left\|u^*v^{1-p}\right\|_{\tilde\Phi,Q}\leq C\max\{1,v^{1-q'}(x)\}\lambda= C\max\{1,v^{1-q'}(x)\}\|u^*\|_{\varphi_p,Q,v^{1-q'}}\]
	and by proceeding similarly as in the proof of Theorem~\ref{teo: acotacion mixta general para T} we can obtain \eqref{eq: teo: F-S para T Hormander - relacion maximales con y sin peso}. This allows us to finish the estimate of $I$ by following similar lines as in page~\pageref{pag: estimacion de I}. For $II$, we use again that $t\lesssim \varphi_p(t)$ combined with the fact that $v^{1-q'}\in A_1$. We also notice that $p'<q'$ since $r>1$, so we get $\Psi(v)\geq v$. By following the same argument as in page~\pageref{pag: estimacion de II} we get the desired bound.
	
	We finish with the estimate of $III$. We denote $A_{j,k}=\{x: 2^{k-1}r_j<|x-x_{Q_j}|\leq 2^{k}r_j\}$, where $r_j=cR\ell(Q_j)/8$ and use the integral representation of $T$ given by \eqref{eq: representacion integral de T}. By combining \eqref{eq: hj integra cero contra v} with the $L^{\Phi}-$Hörmander condition on \eqref{eq: condicion Hormander} $K$ we get
	\begin{align*}
	III&\leq uv\left(\left\{x\in \mathbb{R}^n\backslash\Omega^*: \sum_j\frac{|T(h_jv)|}{v}>\frac{t}{2}\right\}\right)\\
	&\leq \frac{C}{t}\int_{\mathbb{R}^n\backslash\Omega^*}\sum_j|T(h_jv)(x)|u_j^*(x)\,dx\\
	&\leq \frac{C}{t}\sum_j\int_{\mathbb{R}^n\backslash Q_j^*}\left|\int_{Q_j}h_j(y)v(y)(K(x-y)-K(x-x_{Q_j}))\,dy\right|u_j^*(x)\,dx\\
	&\leq \frac{C}{t}\sum_j\int_{Q_j}|h_j(y)|v(y)\int_{\mathbb{R}^n\backslash Q_j^*}|K(x-y)-K(x-x_{Q_j})|u_j^*(x)\,dx\,dy\\
	&\leq \frac{C}{t}\sum_j\int_{Q_j}|h_j(y)|v(y)\sum_{k=1}^{\infty}{\int_{A_{j,k}}}|K(x-y)-K(x-x_{Q_j})|u_j^*(x)\,dx\,dy\\
	&= \frac{C}{t}\sum_j\int_{Q_j}|h_j(y)|v(y)F_j(y)\,dy,
	\end{align*}
	where 
	\[F_j(y)=\sum_{k=1}^{\infty}{\int_{A_{j,k}}}|K(x-y)-K(x-x_{Q_j})|u_j^*(x)\,dx.\]
	We shall prove that there exists a positive constant $C$ such that
	\[F_j(y)\leq CM_{\tilde \Phi}u_j^*(y),\]
	for every $y\in Q_j$. Indeed, by applying generalized Hölder inequality with the functions $\Phi$ and $\tilde \Phi$, since $K\in H_\Phi$ we have that
	\begin{align*}
	F_j(y)&\leq C\sum_{k=1}^\infty (2^kr_j)^n\left\|(K(\cdot-(y-x_{Q_j})-K(\cdot))\mathcal{X}_{A_{j,k}}\right\|_{\Phi, B\left(x_{Q_j}, 2^kr_j\right)}\left\|u_j^*\right\|_{\tilde\Phi,B\left(x_{Q_j}, 2^kr_j\right)}\\
	&\leq C_\Phi M_{\tilde \Phi}u_j^*(y).
	\end{align*}
	Thus, by Lemma~\ref{lema: comparacion peso por caracteristica} we get
	\[III\leq \frac{C}{t}\sum_j\int_{Q_j}|f|v\left(\inf_{Q_j}M_{\tilde\Phi}u_j^*\right)+C\sum_j\int_{Q_j}|f_{Q_j}^v|v\left(\inf_{Q_j}M_{\tilde\Phi}u_j^*\right).\]
	Recall that $M_{\tilde\Phi}\lesssim M_{\varphi_p}\lesssim M_{\varphi_p,v^{1-q'}}$. This allows us to conclude the estimate similarly as we did in page~\pageref{pag: estimacion de III}. The proof is complete when $u$ is bounded. For the general case we can proceed as in the proof of Theorem~\ref{teo: acotacion mixta general para T}.	
\end{proof}

\bibliographystyle{plain}

\def\cprime{$'$}

\end{document}